\documentclass[12pt]{amsart}
\usepackage{amssymb}
\usepackage{amsmath}
\usepackage{amsthm}
\usepackage{framed}
\usepackage{xcolor}
\usepackage{mathtools}
\usepackage{mathrsfs}
\usepackage{tikz-cd, amssymb}
\usetikzlibrary{decorations.pathmorphing}
\usetikzlibrary{shapes.geometric}
\usepackage{ upgreek }
\usepackage{bbm}
\usepackage{enumitem}
\usepackage{stmaryrd}
\usepackage{blindtext}
\usepackage[hidelinks]{hyperref}
\usepackage[margin=3.1cm]{geometry}
\usepackage{verbatim}
\usepackage{multicol}
\usepackage{cancel}
\usepackage{tabularray}
\usepackage{comment}
\usepackage{quiver}
\usepackage{graphicx}

\definecolor{myred}{RGB}{250,170,85}
\definecolor{mygreen}{RGB}{155,250,100}
\definecolor{myyellow}{RGB}{245,215,25}
\definecolor{mypurple}{RGB}{170,85,250}

\definecolor{ygreen}{rgb}{0.12, 0.67, 0.60}
\definecolor{dgreen}{rgb}{0.08, 0.46, 0.03}

\makeatletter
\newcommand\RedeclareMathOperator{%
  \@ifstar{\def\rmo@s{m}\rmo@redeclare}{\def\rmo@s{o}\rmo@redeclare}%
}
\newcommand\rmo@redeclare[2]{%
  \begingroup \escapechar\m@ne\xdef\@gtempa{{\string#1}}\endgroup
  \expandafter\@ifundefined\@gtempa
     {\@latex@error{\noexpand#1undefined}\@ehc}%
     \relax
  \expandafter\rmo@declmathop\rmo@s{#1}{#2}}
\newcommand\rmo@declmathop[3]{%
  \DeclareRobustCommand{#2}{\qopname\newmcodes@#1{#3}}%
}
\@onlypreamble\RedeclareMathOperator
\makeatother

\DeclareMathOperator{\RR}{\mathbb{R}}
\DeclareMathOperator{\CC}{\mathbb{C}}
\DeclareMathOperator{\ZZ}{\mathbb{Z}}

\DeclareMathOperator{\Z2}{\mathbb{Z}/2}
\DeclareMathOperator{\HZ}{\mathrm{H}\mathbb{Z}}
\newcommand{\HZmod}{\mathrm{H}\mathbb{Z} / 2}

\DeclareMathOperator{\KO}{\mathrm{KO}}
\DeclareMathOperator{\ku}{\mathrm{ku}}
\DeclareMathOperator{\ko}{\mathrm{ko}}

\DeclareMathOperator{\KU}{\mathrm{KU}}

\RedeclareMathOperator{\H}{\mathrm{H}}
\DeclareMathOperator{\C}{\mathcal{C}}

\DeclareMathOperator{\id}{\mathrm{id}}

\DeclareMathOperator{\Spin}{\mathrm{Spin}}

\DeclareMathOperator{\st}{\: \big| \:}
\DeclareMathOperator{\BSpin}{\mathrm{BSpin}}

\DeclareMathOperator{\ESpin}{\mathrm{ESpin}}
\DeclareMathOperator{\MSpin}{\mathrm{MSpin}}
\newcommand{\MSpinR}[0]{\mathrm{MSpin}^c_{\RR}}
\newcommand{\MSpinRe}[0]{(\mathrm{MSpin}^c_{\RR})^{e}}
\newcommand{\kuangle}[1]{\ku\langle {#1} \rangle}
\newcommand{\kuRangle}[1]{\ku_{\RR}\langle {#1} \rangle}

\DeclareMathOperator{\SO}{\mathrm{SO}}
\DeclareMathOperator{\U}{\mathrm{U}}
\RedeclareMathOperator{\O}{\mathrm{O}}
\RedeclareMathOperator{\P}{\mathcal{P}}

\DeclareMathOperator{\B}{\mathrm{B}}
\DeclareMathOperator{\E}{\mathrm{E}}
\DeclareMathOperator{\BSO}{\mathrm{BSO}}

\DeclareMathOperator{\Fun}{Fun}
\newcommand{\Sp}{\mathrm{Sp}}
\newcommand{\Sq}{\mathrm{Sq}}
\newcommand{\SpBC}{\mathrm{Sp}^{\mathrm{B}C_2}}
\newcommand{\BC}{\mathrm{B}C_2}

\DeclareMathOperator{\Map}{\mathrm{Map}}
\DeclareMathOperator{\map}{\mathrm{map}}
\RedeclareMathOperator{\c}{\overline{(\:\:)}}

\theoremstyle{definition}
\newtheorem{theorem}{Theorem}[section]
\newtheorem{lemma}[theorem]{Lemma}

\newtheorem{conjecture}[theorem]{Conjecture}
\newtheorem{corollary}[theorem]{Corollary}
\newtheorem{question}[theorem]{Question}
\newtheorem{proposition}[theorem]{Proposition}

\numberwithin{subcase}{case}
\newtheorem{remark}[theorem]{Remark}

\newtheorem{definition}[theorem]{Definition}

\hypersetup{
    colorlinks=true,
    linkcolor=dgreen,
    citecolor = dgreen,
    urlcolor=dgreen,
}

\title{The homotopy fixed points of Real spin bordism}

\author{ Hassan H. Abdallah}
\address{Department of Mathematics, Wayne State University, Detroit, MI, USA}
\email{hassan@wayne.edu} 

\author{Yigal Kamel}
\address{Department of Mathematics, University of Illinois at Urbana-Champaign, Urbana, IL, USA}
\email{ykamel2@illinois.edu}

\begin{document}

\vspace{0mm}

\begin{abstract}
We show that the 2-local splitting of spin$^c$ bordism by Anderson--Brown--Peterson and Stong refines to a $C_2$-equivariant map in the category of spectra with $C_2$-action from Real spin bordism to a sum of (higher) connective covers of $\ku_{\RR}$ and suspensions of mod 2 Eilenberg--Mac Lane spectra. We use this to deduce a corresponding 2-local splitting of the homotopy fixed points of Real spin bordism. We also discuss prospects that arise in the genuine setting.
\end{abstract}

\maketitle

\vspace{-5mm}

\tableofcontents

\section{Introduction}

In their seminal paper \cite{ABPspin67}, Anderson--Brown--Peterson showed that spin manifolds are determined up to spin bordism by $\KO$-characteristic numbers and Stiefel--Whitney numbers by providing a 2-local splitting of the spin bordism spectrum, $\MSpin$, in terms of (higher) connective covers of $\ko$ and suspensions of $\HZmod$ \eqref{ABP.map.r}. Stong \cite{Stong68} adapted these constructions to yield a similar splitting of the spin$^c$ bordism spectrum, $\MSpin^c$, in terms of covers of $\ku$ and $\HZmod$ \eqref{ABP.map}. The main result of this paper refines this splitting of spin$^c$ bordism to a $C_2$-equivariant splitting of the Real spin bordism spectrum, $\MSpin^c_{\RR}$, of Halladay and the second author \cite{HK24}, in the category of spectra with $C_2$-action.

\begin{theorem}\label{equiv.ABP}
There is a $C_2$-equivariant map of spectra with $C_2$-action,
\begin{equation}\label{Real.ABP.map}
F^c_{\RR} : (\MSpin^c_{\RR})^e \xrightarrow{\big(\bigvee_I f^I_{\RR} \big)\times\big(\bigvee_z f^z_{\RR}\big)} \bigvee_{I \in \mathcal{P}} \ku_{\RR} \langle 4|I| \rangle^e \vee \bigvee_{z \in Z} \Sigma^{|z|}\HZmod,
\end{equation}
whose underlying spectrum map is the 2-local splitting of $\MSpin^c$ of Anderson--Brown--Peterson \cite{ABPspin67} and Stong \cite{Stong68}. 
\end{theorem}

There are three key ingredients that go into defining the $\mathrm{K}$-theory components of the classical splittings of Anderson--Brown--Peterson \eqref{ABP.map}:
\begin{enumerate}
    \item the spin and spin$^c$ orientations, $\varphi: \MSpin \to \KO$ and $\varphi^c : \MSpin^c \to \KU$, of Atiyah--Bott--Shapiro \cite{ABS};
    \item the construction of the $\KO$-valued characteristic classes, $\pi^I_r \in \KO^0(\BSO)$, in Anderson--Brown--Peterson's prior work on $\mathrm{SU}$-bordism \cite{ABPsu66}; 
    \item the determination of the filtration level of the characteristic classes $\pi^I_r$ in $\KO^0$ and their complexifications in $\KU^0$ in order to lift to appropriate higher connective covers of $\ko$ and $\ku$, respectively. 
\end{enumerate}

The remaining $\HZmod$ components of the splittings are obtained from a detailed analysis of the mod 2 cohomology of $\MSpin$ and $\MSpin^c$, respectively.  For a more detailed review of the construction, see Section \ref{sec.rev.ABP}.

\vspace{4mm}

The bulk of the proof of Theorem \ref{equiv.ABP} can be found in Section \ref{sec.equiv.ABP}. Section \ref{sec.Kequiv} consists of showing that each of the steps listed above can be carried out with Reality in the Borel equivariant setting. Incorporating Reality in step (1) is the content of the \textit{Real spin orientation} of \cite{HK24}. Adapting step (2) uses the observation that the $\KU$-characteristic classes responsible for the splitting of $\MSpin^c$ are defined as complexifications of $\KO$-characteristic classes. This step also requires a better understanding of the Real structure on $\BSpin^c$, which we address in Section \ref{sec.prelim}. Adapting step (3) involves a simple, but subtle, computation to show that the characteristic classes $\pi^I$ lift to the appropriate connective covers of $\ku_{\RR}^e \in \SpBC$, despite the fact that they do not lift to the corresponding genuine connective covers of $\ku_{\RR} \in \Sp^{C_2}$ (see Remark \ref{rem.noliftko}). Section \ref{sec.mod2equiv} is devoted to showing that each of the $\HZmod$ components of the classical splitting descend to the mod 2 Borel cohomology of $\MSpinR$.  

\vspace{4mm}

In Section \ref{sec.HFP}, we apply Theorem \ref{equiv.ABP} to obtain a corresponding 2-local splitting of the homotopy fixed points of Real spin bordism. 

\begin{theorem}[Corollary \ref{cor.homotopy.fixedpoints}]\label{homotopy.fixedpoints}
    There is a 2-local equivalence, 
    $$
    (\MSpinR)^{hC_2} \to \bigvee_{I \in \mathcal{P}}\kuRangle{4|I|}^{hC_2} \vee \bigvee_{z \in Z}( \Sigma^{|z|}\HZmod)^{hC_2}.
    $$
\end{theorem}

This relies on the facts that a map of spectra with $C_2$-action that induces an equivalence on underlying spectra is automatically an equivalence of spectra with $C_2$-action, and that under mild hypotheses, 2-localization commutes with homotopy fixed points. Theorem \ref{homotopy.fixedpoints} then allows for a computation of the homotopy groups of $(\MSpinR)^{hC_2}$.

\begin{theorem}[Corollary \ref{cor.homotopy.groups}]\label{thm.HFP.groups} The homotopy groups of the $C_2$-homotopy fixed points of Real spin bordism are given by
$$
\pi_*(\MSpinR)^{hC_2} \cong \bigoplus_{I \in \mathcal{P}} (\pi_*\ko\langle 4 |I| \rangle \oplus \bigoplus_{m\geq1} \mathbb{Z}/2\{\delta^m_I\}) \oplus \bigoplus_{z \in Z} \H^{-*+|z|}(\BC ; \Z2), 
$$    

where $\Z2\{\delta^m_I\}$ denotes a factor of $\Z2$ generated by an element $\delta^m_I$ with degree $|\delta^m_I| = 4|I| - 4m$ when $|I|$ is even, and $|\delta^m_I| = 4|I| - 2 - 4m$ when $|I|$ is odd. 
\end{theorem}

In Section \ref{sec.genuine}, we discuss an obstruction to refining the construction of Anderson--Brown--Peterson to the genuine setting. We introduce new $C_2$-spectra, $\kuRangle{4n,2}$, that circumvent this obstruction and potentially play a role in a genuine splitting of Real spin bordism. 

\subsection{Acknowledgments} We would like to thank Bob Bruner, 
Kiran Luecke, Fredrick Mooers, 
Andrew Salch, Brian Shin, Vesna Stojanoska, Vivasvat Vatatmaja, and Alex Waugh for very helpful conversations. We give special thanks to Zach Halladay for an enriching collaboration on Real spin bordism and for countless helpful conversations on equivariant stable homotopy theory in general; his impact is present throughout the paper.  




\section{Preliminaries}\label{sec.prelim}

In this section, we briefly recall some facts that we need about equivariant homotopy theory and Real spin bordism. 

\subsection{Elements of $C_2$-equivariant homotopy theory}

Throughout this paper, we will primarily work in the ``Borel'' $C_2$-equivariant setting; that is, functors on $\BC$.

\vspace{3mm}

Let $\mathcal{C}$ be an $\infty$-category. The $\infty$-category of \textit{$\mathcal{C}$-objects with $C_2$-action} is the functor category $\mathcal{C}^{\BC} := \Fun(\BC,\mathcal{C}).$ The \textit{underlying object} functor $U : \C^{\BC} \to \C$ is the evaluation at the single object of $\BC$. Let $A$ and $B$ be objects in $\C^{\BC}$. We say that a map $f : U(A) \to U(B) \in \C$  \textit{refines to a $C_2$-equivariant map} $\alpha : A \to B \in \C^{\BC}$ if $U(\alpha) = f$. 

\begin{proposition}\label{prop.underlying.equiv}
    A morphism $f:X \to Y$ in $\mathcal{C}^{\BC}$ is an equivalence if and only if the induced map, $U(f) : U(X) \to U(Y)$, on underlying objects is an equivalence in $\mathcal{C}$.
\end{proposition}
\begin{proof}
    In any functor $\infty$-category, $\Fun(\mathcal{C},\mathcal{D})$, a morphism is invertible if and only if all of its components are invertible as morphism in $\mathcal{D}$ (see Corollary 3.5.12 of \cite{Cisinski_19}).
\end{proof}

Thus, if an equivalence in $\mathcal{C}$ refines to a $C_2$-equivariant map, then it refines to an equivalence in $\mathcal{C}^{\BC}$. The \textit{homotopy fixed points} of an object with $C_2$-action, $X \in \C^{\BC}$, is the limit, $$X^{hC_2} := \lim_{\BC}X \in \C.$$  The homotopy fixed point functor is right adjoint to the constant diagram functor $\iota : \C \to \C^{\BC}$, 
\begin{equation}\label{HFP.adj}
\Map_{\C^{\BC}}(\iota X, Y) \simeq \Map_{\C}(X, Y^{hC_2}).
\end{equation}
Similarly, the \textit{homotopy orbits}, $(\:\:)_{hC_2} : \mathcal{C}^{\BC} \to \mathcal{C}$ is left adjoint to $\iota$,
$$
\Map_{\C}(X_{hC_2}, Y) \simeq \Map_{\C^{\BC}}(X, \iota Y),
$$
and is given by the colimit over $\BC$. We will often denote $\iota X$ by $X$. In particular, we are most interested in the cases when $\mathcal{C} = \mathcal{S}$, the $\infty$-category of spaces, and $\mathcal{C} = \Sp$, the $\infty$-category of spectra. In the case of spectra, we will also have occasion to consider the $\infty$-category, $\Sp^{C_2}$, of genuine $C_2$-spectra, which (for concreteness) we take to be the $\infty$-categorical localization of the model category of orthogonal $C_2$-spectra indexed by a complete $C_2$-universe (e.g. \cite{HHRbook}). In this setting, there is a \textit{(genuine) $C_2$-fixed points} functor  $(\:\:)^{C_2} : \Sp^{C_2} \to \Sp$ which also has a left adjoint $\mathrm{infl} : \Sp \to \Sp^{C_2}$, $$\Map_{\Sp^{C_2}}(\mathrm{infl} X, Y) \simeq \Map_{\Sp}(X, Y^{C_2}),$$ such that $(\mathrm{infl}X)^e \simeq \iota X$, where $(\:\:)^e : \Sp^{C_2} \to \SpBC$ is the \textit{underlying spectrum with $C_2$-action} functor. There is a corresponding induced map from fixed points to homotopy fixed points, $$X^{C_2} \to (X^{e})^{hC_2} =: X^{hC_2},$$
so that given a $C_2$-spectrum, $Y$, a map of spectra $X \to Y^{C_2}$ induces an equivariant map, $X \to Y^e$, in $\SpBC$. 

\vspace{3mm}

There is another type of fixed points functor $(\:\:)^{gC_2} : \Sp^{C_2} \to \Sp$, called the \textit{geometric fixed points}, that we will use. The main feature of this functor that we need is its role in the Tate diagram,  
$$
    \begin{tikzcd}[column sep = 4mm, row sep = 9mm]
        X_{hC_2} \arrow[rrr] \arrow[d,"=\:"'] &&& X^{C_2}  \arrow[rrr] \arrow[d]  \arrow[drrr, phantom, "\ulcorner", near end] &&& X^{gC_2} \arrow[d] \\
        X_{hC_2} \arrow[rrr] &&& X^{hC_2}  \arrow[rrr] &&& X^{tC_2},
    \end{tikzcd}
$$
where the rows are cofiber sequences and the square on the right is cartesian.

\subsection{Real bundle theory}\label{sec.equiv.bundle}

Given a $C_2$-action on a group $G$, there is an associated equivariant bundle theory, as developed in \cite{LashofMay86_EquivBundle} and \cite{GuillouMayMerling17_equivBundleCat}. The central objects of study there are called principal $(G,G \rtimes C_2)$-bundles, which we call Real principal $G$-bundles instead. The goal of this section is to identify a convenient model for the underlying space with $C_2$-action of the classifying space for Real $G$-bundles (Proposition \ref{equiv.classifyingspace.action}), which will be applied in Section \ref{sec.RealSpinBordism} to $G = \Spin^c(n)$ with its complex conjugation action.

\begin{definition}\label{def.RealGbund}
    A \textit{Real principal $G$-bundle} is a (topological) principal $G$-bundle, $E \to B$, together with $C_2$-actions on $E$ and $B$, such that the $G$-action, $G \times E \to E$, on $E$ and the bundle map, $E \to B$, are both $C_2$-equivariant.
\end{definition}

\begin{remark}
    Equivalently, a Real principal $G$-bundle is a principal $G$-bundle together with an extension of the $G$-action on the total space to $G \rtimes C_2$. Then the $C_2$-action on the base is defined as the induced action on $E/G \cong B$.
\end{remark}

\begin{definition}
    A Real principal $G$-bundle, $p: E \to B$, is \textit{universal}, if pulling back $p$ induces a natural bijection,
    $$
    [X,B]_{C_2} \cong \{\text{Real principal $G$-bundles on } X \} /\text{isomorphism}.
    $$
\end{definition}

\begin{proposition}[\cite{May_EquivariantBook}]
    There exists a universal Real principal $G$-bundle, $$\E(G;C_2) \to \B(G;C_2).$$    
\end{proposition}

\begin{definition} Let $G$ be a compact Lie group.
\begin{itemize}
    \item Let $\mathcal{E}G$ denote the topological category with object space $G$ and with a unique (iso)morphism between every pair of objects, where the morphism space is topologized as the product  $G \times G$. For example, 
    $$
    \mathcal{E}C_2 = \begin{tikzcd}
        e \arrow[r,bend left = 33] \arrow[from=1-1, to=1-1, loop, in=215, out=145, distance=8mm] & c \arrow[l, bend left = 33] \arrow[from=1-2, to=1-2, loop, in=35, out=325, distance=8mm]
    \end{tikzcd}.
    $$
    \item Let $\mathcal{B}G$ denote the topological category with one object and with morphism space equal to $G$.
    \end{itemize}
\end{definition}

There is a continuous functor $\mathcal{E}G \to \mathcal{B}G$ given on morphisms by the map $G \times G \to G$ defined by $(g, h) \mapsto gh^{-1}$. 

\begin{proposition}
    The map $|\mathcal{E}G| \to |\mathcal{B}G|$ is a universal principal $G$-bundle. 
\end{proposition}
\begin{proof}
    It is immediately clear that $|\mathcal{E}G|$ is contractible with free $G$-action, and that on categories, $\mathcal{E}G / G \cong \mathcal{B}G$. While geometric realization fails to commute with taking quotients in general, in this particular situation, there is a canonical identification $|\mathcal{E}G|/G \cong |\mathcal{B}G| \simeq \B\! G$, as discussed in \cite{GuillouMayMerling17_equivBundleCat} and \cite{nlabCatBundles}. 
\end{proof}

When $G$ has a $C_2$-action, then there is an induced $C_2$-action on $\mathcal{B}G$ by acting on the morphisms in the same way that $C_2$ acts on $G$, and a $C_2$-action on $\mathcal{E}G$ by acting on objects as $C_2$ acts on $G$ (which determines the action on morphisms). One nice feature of this $C_2$-action on $|\mathcal{B}G| \simeq \B\!G$ is that its underlying object of $\mathcal{S}^{\BC}$ is evidently given by the composite $$\BC \xrightarrow{G} \mathrm{Groups} \xrightarrow{\B} \mathcal{S}.$$

\begin{proposition}
    The map $|\mathcal{E}G| \to |\mathcal{B}G|$ is a Real principal $G$-bundle. 
\end{proposition}
\begin{proof}
    This follows from the fact that $C_2$ acts on $G$ by group homomorphisms. 
\end{proof}

By the definition of $\B(G;C_2)$, the Real $G$-bundle $|\mathcal{E}G| \to |\mathcal{B}G|$ is classified by a $C_2$-equivariant map $$f_G : |\mathcal{B}G| \to \B(G;C_2).$$  

\begin{proposition}\label{equiv.classifyingspace.action}
    The map $f_G : |\mathcal{B}G| \to \B(G;C_2)$ is an equivalence in $\mathcal{S}^{\BC}$.
\end{proposition}
\begin{proof}
    Nonequivariantly, $|\mathcal{E}G| \to |\mathcal{B}G|$ is a universal $G$-bundle, so $f_G$ is an equivalence on underlying spaces. Since $f_G$ is $G$-equivariant, by Proposition \ref{prop.underlying.equiv}, $f_G$ is an equivalence in $\mathcal{S}^{\BC}$.
\end{proof}

Thus, the underlying $C_2$-action on the classifying space, $\B(G;C_2)$, can be described by applying the functor $\B$ to the $C_2$-action on $G$.

\subsection{Real spin bordism}\label{sec.RealSpinBordism} In this section, we briefly recall the Real spin bordism spectrum of \cite{HK24}, and we reformulate some of its properties in a way that is more convenient for our purposes in this paper. First, we recall the main result of \cite{HK24}.

\begin{theorem}[Halladay--Kamel \cite{HK24}]\label{Real.orientation}
    There is a genuine $C_2$-ring spectrum $\MSpin^c_{\RR}$, called \textit{Real spin bordism}, and a ring map $\varphi^c_{\RR} : \MSpinR \to \KU_{\RR}$, whose underlying map is the spin$^c$ orientation, $\varphi^c : \MSpin^c \to \KU$. Furthermore, there exists a natural ring map $\MSpin \to (\MSpinR)^{C_2}$. 
\end{theorem}

In this paper, we are primarily interested in the underlying spectrum with $C_2$-action, $(\MSpinR)^e \in \SpBC$ of Real spin bordism. We now review the relevant actions in this context. Let $\Spin^c_{\RR}(n)$ denote the group with $C_2$-action defined by $$\displaystyle \Spin^c_{\RR}(n) = \Spin(n) \underset{\{\pm 1\}}{\times} \U_{\RR}(1),$$ where $\Spin(n)$ is given the trivial $C_2$-action, and $\U_{\RR}(1)$ is the group $\U(1)$ equipped with the $C_2$-action given by complex conjugation. In particular, we have equivariant short exact sequences,

\vspace{-5mm}

\begin{equation}\label{equiv.ses}
1 \to \U_{\RR}(1) \to \Spin^c_{\RR}(n) \to \SO(n) \to 1,
\end{equation}

\vspace{-8mm}

\begin{equation}\label{equiv.ses.2}
1 \to \Spin(n) \to \Spin^c_{\RR}(n) \to \U_{\RR}(1) \to 1,
\end{equation}

where $\SO(n)$ is given the trivial action. The construction of the Real spin bordism spectrum in \cite{HK24} uses a particular topological model for the total space of the universal $\Spin^c(n)$-bundle, $\E_{\mathrm{J}} \simeq \ESpin^c(n)$ (adapted from \cite{Joachim} and denoted $\displaystyle \mathrm{U}^{\mathrm{even}}_{\RR^n}$ in \cite{HK24}), equipped with a $C_2$-action that satisfies the following properties.

\vspace{2mm}

\begin{enumerate}[itemsep=1mm]
    \item The $\Spin^c(n)$-action induces a $C_2$-equivariant map $\Spin^c_{\RR}(n) \times \E_{\mathrm{J}} \to \E_{\mathrm{J}}$. 
    \item The $C_2$-fixed point space of $\E_{\mathrm{J}}$ is contractible, $\E_{\mathrm{J}}^{C_2} \simeq *$. 
\end{enumerate}

\vspace{2mm}

Property (1) is then used to define a $C_2$-space  $\B_{\mathrm{J}} \simeq \BSpin^c(n)$ as the quotient $\B_{\mathrm{J}} = \E_{\mathrm{J}}/\Spin^c(n)$, which then implies that the quotient map $\E_{\mathrm{J}} \to \B_{\mathrm{J}}$ is a Real $\Spin^c_{\RR}(n)$-bundle, in the sense of Definition \ref{def.RealGbund}. However, an explicit description of the $C_2$-action on $\B_{\mathrm{J}}$ is not given in \cite{HK24}. In this paper, we are interested in this $C_2$-action, but we are not interested in the specific model given in \cite{HK24}. We present here a different way to construct the relevant $C_2$-action on $\BSpin^c(n)$ that is clearer for our purposes. For this, we use the ideas of Section \ref{sec.equiv.bundle}.

\vspace{3mm}

Let $\ESpin^c_{\RR}(n) \to \BSpin^c_{\RR}(n)$ denote the universal Real principal $\Spin^c_{\RR}(n)$-bundle. Since $\E_{\mathrm{J}} \to \B_{\mathrm{J}}$ is a Real $\Spin^c_{\RR}(n)$-bundle, it is determined by a $C_2$-equivariant map, 
$$
   f_{\mathrm{J}}: \B_{\mathrm{J}} \to \BSpin^c_{\RR}(n). 
$$

\begin{proposition}\label{BHK.equiv}
    The map $f_{\mathrm{J}}:\B_{\mathrm{J}} \to \BSpin^c_{\RR}(n)$ is an equivalence in $\mathcal{S}^{\BC}$.
\end{proposition}
\begin{proof}
    Forgetting the $C_2$-actions, $\E_{\mathrm{J}}\to \B_{\mathrm{J}}$ is a universal $\Spin^c(n)$-bundle which is classified by the map $f_{\mathrm{J}}$. Thus, $f_{\mathrm{J}}$ must be an equivalence of underlying spaces. By Proposition \ref{prop.underlying.equiv}, it is an equivalence in $\mathcal{S}^{\BC}$. 
\end{proof}

Putting together Propositions \ref{equiv.classifyingspace.action} and \ref{BHK.equiv}, we see that in $\mathcal{S}^{\BC}$, we have equivalences
$$
\B_{\mathrm{J}} \simeq \BSpin^c_{\RR}(n) \simeq |\mathcal{B}\Spin^c(n)|. 
$$
Throughout the paper, we will use the notation $\BSpin^c_{\RR}(n) \in \mathcal{S}^{\BC}$ for this object (and $\BSpin^c_{\RR} \in \mathcal{S}^{\BC}$ for the colimit over $n$), and freely use the fact that it can be obtained either from the constructions of \cite{HK24} or as the composite $$\BC \xrightarrow{\Spin^c_{\RR}(n)} \mathrm{Groups} \xrightarrow{\B} \mathcal{S}.$$

\section{Review of the Anderson--Brown--Peterson splittings}\label{sec.rev.ABP}

First, we briefly recall the construction of the spin${}^c$ version of the Anderson--Brown--Peterson map \cite{ABPspin67} (see \cite{Stong68} or \cite{BuchananMckean_KSpMSpinh} for details).  

\vspace{3mm}

Let $\pi^i_r : \BSO \to \KO$ denote the $i$-th $\KO$-Pontrjagin class, as defined in \cite{ABPsu66}, and let $\pi^i \in \KU^0(\BSO)$ denote its complexification,
$$
\pi^i : \BSO \xrightarrow{\pi^i_r} \KO \xrightarrow{\otimes \CC} \KU,
$$
(where $\otimes \CC$ denotes the inclusion of fixed points), as well as it's pullback to $\BSpin^c$, $\pi^i \in \KU^0(\BSpin^c)$. Let  $\mathcal{P} = \{(i_1,...,i_k) \st k\geq 0, i_l \in \ZZ_{\geq 1})\}$ be the set of integer partitions, including the empty partition $(\:\:)$ of 0. Given a partition $I = (i_1,...,i_k)$, let $|I| = \sum_{l=1}^k i_l$ be the underlying integer that $I$ partitions. Given $I = (i_1, \dots, i_k) \in \P$, let 
\[
\begin{aligned}
&  &&\pi^I_r = \pi_r^{i_1}\dots\pi_r^{i_k} : \BSO \to \KO, \\
&\text{and} &&\pi^I = \pi^{i_1}\dots\pi^{i_k} : \BSO \to \KU.
\end{aligned}
\]

\begin{theorem}[Anderson--Brown--Peterson \cite{ABPspin67} and Stong \cite{Stong68}]\label{thm.filtration} 
Let $I = (i_1,...,i_k) \in \mathcal{P}$, and let $$n_I = \begin{cases} 4|I|, & |I| \text{ is even} \\ 4|I| -2, & |I| \text{ is odd}. \end{cases}$$ Then there exist lifts,
        $$
        \begin{tikzcd}[row sep = 1mm]
            & \ko\langle n_I \rangle \arrow[dd] && &\ku\langle 4|I| \rangle \arrow[dd] \\
            &&\text{and}&&\\
            \BSO \arrow[r,"\pi^I_r"'] \arrow[uur,dashed, "\widetilde{\pi}^I_r"] &\KO, && \BSO \arrow[r,"\pi^I"'] \arrow[uur,dashed, "\widetilde{\pi}^I"] &\KU.
        \end{tikzcd}
        $$
\end{theorem}

Recall that for any $n,m \in \ZZ_{\geq 0}$, the multiplication on $\KU$ lifts to a map 
$$
\mu:\ku\langle n \rangle \wedge \ku \langle m \rangle \to \ku\langle n+m \rangle,
$$
and let $\varphi : \MSpin \to \ko$ and $\varphi^c : \MSpin^c \to \ku$ denote the Atiyah--Bott--Shapiro orientations \cite{ABS}.

\begin{definition}[\cite{ABPspin67}, \cite{Stong68}]\label{def.fI}
Define $f^I_r : \MSpin \to \ko\langle n_I \rangle$ and $f^I : \MSpin^c \to \ku\langle 4|I| \rangle$ to be the composites,
$$
\begin{aligned}
f^I_r &: \MSpin \xrightarrow{\Delta} \MSpin \wedge \BSpin \xrightarrow{\varphi \wedge \widetilde{\pi}^I_r} \ko \wedge \ko \langle n_I \rangle \xrightarrow{\mu} \ko \langle n_I \rangle, \\
f^I &: \MSpin^c \xrightarrow{\Delta} \MSpin^c \wedge \BSpin^c \xrightarrow{\varphi^c \wedge \widetilde{\pi}^I} \ku \wedge \ku \langle 4|I| \rangle \xrightarrow{\mu} \ku \langle 4|I| \rangle, 
\end{aligned}
$$
where $\Delta$ is the Thom diagonal. 
\end{definition} 

Lastly, let $\mathcal{P}_1 = \{(i_1,...,i_k) \in \mathcal{P} \st  i_l \geq 2)\}$ be the set of partitions that do not contain 1 as a summand.

\begin{theorem}[Anderson--Brown--Peterson \cite{ABPspin67} and Stong \cite{Stong68}]
    There exist generators of free $\mathcal{A}$-module summands $Z_r \subset \H^*(\MSpin;\Z2)$ and $Z \subset \H^*(\MSpin^c;\Z2)$, such that the maps 
    \begin{equation}\label{ABP.map.r}
    F: \MSpin \xrightarrow{(\bigvee_I f^I_r) \vee (\bigvee_z f^{z}_r)} \bigvee_{I \in \P_1} \ko \langle n_I \rangle  \vee \bigvee_{z \in Z_r} \Sigma^{|z|}\HZmod,
    \end{equation}
    and
    \begin{equation}\label{ABP.map}
    F^c: \MSpin^c \xrightarrow{(\bigvee_I f^I) \vee (\bigvee_z f^z)} \bigvee_{I \in \P} \ku\langle 4|I| \rangle  \vee \bigvee_{z \in Z} \Sigma^{|z|}\HZmod,
    \end{equation}
    are 2-local equivalences.
\end{theorem}

\section{The Anderson--Brown--Peterson map is $C_2$-equivariant}\label{sec.equiv.ABP}

In this section, we refine the constructions of the previous section to incorporate $C_2$-equivariance and prove Theorem \ref{equiv.ABP}.  Let $\MSpinRe \in \SpBC$ denote the underlying spectrum with $C_2$-action of the Real spin bordism spectrum, $\MSpinR$, constructed in \cite{HK24}. Similarly, let $\BSpin^c_{\RR} \in \mathcal{S}^{\BC}$ be as in Section \ref{sec.RealSpinBordism}. 

\subsection{Equivariance of the $\mathrm{K}$-theory components}\label{sec.Kequiv}

\begin{proposition}\label{Real.BSpinc.BSO}
    The map $\BSpin^c \to \BSO$ refines to a $C_2$-equivariant map $$\BSpin^c_{\RR} \to \BSO,$$ for the trivial $C_2$-action on $\BSO$.
\end{proposition}
\begin{proof}
Recall the equivariant short exact sequence \eqref{equiv.ses}, 
$$
1 \to \U_{\RR}(1) \to \Spin^c_{\RR}(n) \to \SO(n) \to 1.
$$
Applying the functor $\B$ to the above sequence, gives a map $\BSpin^c_{\RR}(n) \to \BSO(n)$ in $\mathcal{S}^{\BC}$ for the trivial action on $\BSO(n)$.
\end{proof}

\begin{proposition}\label{piI.equiv}
    The map $\pi^I : \BSpin^c \to \KU$ refines to a $C_2$-equivariant map $$\pi^I_{\RR} : \BSpin^c_{\RR} \to \KU_{\RR}^e.$$
\end{proposition}
\begin{proof}
    First note that $\pi^I : \BSO \to \KU$ lifts to a $C_2$-equivariant map $\BSO \to \KU_{\RR}^e$ with respect to the trivial $C_2$-action on $\BSO$, using the adjunction in \eqref{HFP.adj} and the fact that $\pi^I$ factors through $\KO\simeq \KU_{\RR}^{C_2} \simeq \KU_{\RR}^{hC_2}$. Composing with the map of Proposition \ref{Real.BSpinc.BSO} yields the desired map.
\end{proof}

Let $\ku_{\RR}$ denote the equivariant connective cover of $\KU_{\RR}$. More generally, let $\kuRangle{n}$ denote the $n$th equivariant connective cover of $\KU_{\RR}$, as in Section 3.4 of \cite{BrunerGreenlees10_connRealK}, with $$\kuRangle{n}^{C_2} \simeq \ko\langle n \rangle.$$

The underlying spectrum with $C_2$-action, $\kuRangle{n}^e \in \SpBC$, can be described via postcomposition of $\KU_{\RR}^e$ with the $n$th connective cover functor, $\tau_n : \Sp \to \Sp_{\geq n}$,  $$\kuRangle{n}^e : \BC \xrightarrow{\KU_{\RR}^e} \Sp \xrightarrow{\tau_n} \Sp_{\geq n} \hookrightarrow\Sp.$$

\begin{proposition}\label{prop.equiv.lift.cover}
    The map $\widetilde{\pi}^I : \BSpin^c \to \ku\langle 4|I| \rangle$ refines to a $C_2$-equivariant map $$\widetilde{\pi}^I_{\RR} : \BSpin^c_{\RR} \to \kuRangle{4|I|}^e.$$
\end{proposition}
\begin{proof}
By Theorem \ref{thm.filtration}, when $|I|$ is even, the map $\pi^I_r : \BSO \to \KO$ lifts to $\smash{\widetilde{\pi}}^I_r: \BSO \to \ko\langle 4|I| \rangle$, so $\widetilde{\pi}^I_R$ is adjoint to a map of genuine $C_2$-spectra, $\widetilde{\pi}^I_{\RR}:\BSO \to \kuRangle{4|I|}$. When $|I|$ is odd, $\pi^I_r$ lifts to $\widetilde{\pi}^I_r:\BSO \to \ko \langle 4|I|-2 \rangle$, which gives an equivariant map, $\widetilde{\pi}^I_{\RR}:\BSO \to \kuRangle{4|I|-2}$. While this map does not genuinely lift to $\kuRangle{4|I|}$ (see Remark \ref{rem.noliftko}), we now show that it does lift to $\kuRangle{4|I|}^e$ in the category of spectra with $C_2$-action. The underlying spectrum with $C_2$-action, $\kuRangle{4|I|}^e \in \SpBC$, is the homotopy fiber of the map $$\kuRangle{4|I|-2}^e \to \Sigma^{4|I|-2}\HZ_{\sigma},$$ where $\HZ_{\sigma} \in \SpBC$ denotes $\HZ$ with the sign action. Using the homotopy fixed points spectral sequence for  $\Sigma^{4|I|-2}\HZ_{\sigma}$, we find that 
$$
\pi_n((\Sigma^{4|I|-2}\HZ_{\sigma})^{hC_2}) = \begin{cases}
    \Z2, & n \text{ odd and } n \leq 4|I|-3 \\ 
    0, & n \text{ even or } n > 4|I| -3.
\end{cases}
$$
Thus, since $\ko\langle 4|I| -2 \rangle \in \Sp_{\geq 4|I| - 2}$ and $(\Sigma^{4|I|-2}\HZ_{\sigma})^{hC_2} \in \Sp_{\leq 4|I| -3}$, we see that
$$
\begin{aligned}
\Map_{\SpBC}(\ko\langle 4|I|-2 \rangle, \Sigma^{4|I|-2}\HZ_{\sigma}) \: \simeq \: &\Map_{\Sp}(\ko\langle 4|I|-2 \rangle,(\Sigma^{4|I|-2}\HZ_{\sigma})^{hC_2}) \\
 \: \simeq \: & *.
\end{aligned}
$$
So the map $\ko\langle 4|I|-2 \rangle \to \ku_{\RR}\langle 4|I|-2 \rangle^e$ equivariantly factors through the fiber, 
\begin{equation}\label{lift.ko42}
\begin{tikzcd}
    & \kuRangle{4|I|}^e \arrow[d] \\
   \ko\langle 4|I|-2 \rangle \arrow[r] \arrow[ur,"\exists \: c"]  & \kuRangle{4|I|-2}^e \arrow[r] & \Sigma^{4|I|-2}\HZ_{\sigma}
\end{tikzcd}
\end{equation}
Thus, the composite $$\BSpin^c_{\RR} \to \BSO \xrightarrow{\widetilde{\pi}^I_r} \ko\langle 4|I|-2 \rangle \xrightarrow{c} \kuRangle{4|I|}^e$$ is the desired lift.
\end{proof}

\begin{remark}\label{rem.noliftko}
    Despite the fact that $\pi^I : \BSO \to \KU$ factors through $\KO$ and $C_2$-equivariantly lifts to $\kuRangle{4|I|}^e$ in $\SpBC$, $\pi^I$ does \textit{not} lift to $\ko\langle 4|I| \rangle$ when $|I|$ is odd. If it did, then the pullback of $\pi^I_r$ to $\BSpin$ would also lift to $\ko\langle 4|I| \rangle$, but the filtration level of $\pi^I_r \in \KO^0(\BSpin)$ is precisely $4|I| - 2$ (see Theorem 2.1 of \cite{ABPspin67} and page 314 of \cite{Stong68}). This is not a contradiction, since the homotopy fixed points of $\kuRangle{4|I|}$ is not equivalent to $\ko\langle 4|I| \rangle$ (see Proposition \ref{HFP.ku.covers}).  
\end{remark}

\begin{proposition}
    The multiplication map $\mu : \kuangle{n} \wedge \kuangle{m} \to \kuangle{n+m}$ refines to a $C_2$-equivariant map $$\mu_{\RR} : \kuRangle{n}^e \wedge \kuRangle{m}^e \to \kuRangle{n+m}^e.$$ 
\end{proposition}
\begin{proof}
    Consider the following diagram of $\infty$-categories which decomposes $\mu$ after pulling back along $* \to \BC$.
    \vspace{-1mm}
    $$
\begin{tikzcd}[column sep = 1.5cm, row sep = 1cm]
    * \arrow{r} &\BC  \arrow[r, "\KU_{\RR}^e \otimes \KU_{\RR}^e"] \arrow[dr,"\kuRangle{n}^e \otimes \kuRangle{m}^e"'] \arrow[rr, bend left = 35,"\KU_{\RR}^e"{name = D}] \arrow[r,phantom, from =2-2, to = 1-3,yshift = 3mm, xshift=6mm,"="] & \Sp \otimes \Sp \arrow[r,"\wedge"] \arrow[d, "\tau_n \otimes \tau_m"] \arrow[u, Rightarrow, from=2-3, to=1-3, yshift = 1.6cm,shorten = 1.3mm, "\phantom{a}\mu_{\KU_{\RR}^e}"'{xshift = -1mm} ] & \Sp \arrow[d,"\tau_{n+m}"] \\
    \phantom{-} &\phantom{-} & \Sp_{\geq n} \otimes \Sp_{\geq m} \arrow[r, shorten= -2mm, xshift = 1mm, "\wedge"] \arrow[u, Rightarrow, from=2-3, to=1-4, yshift = -0.5mm, ,shorten = 5mm, "\alpha"{xshift = 0mm} ] & \phantom{-}\Sp_{\geq n+m}, 
\end{tikzcd}
    $$
    
where $\mu_{\KU_{\RR}^e} : \KU_{\RR}^e \wedge \KU_{\RR}^e \to \KU_{\RR}^e$ is the multiplication on $\KU_{\RR}^e$. The map $\alpha : X\langle n \rangle \wedge Y\langle m \rangle \to (X \wedge Y)\langle n+m \rangle$ is natural as follows. Let $\epsilon_n : \iota \tau_n \Rightarrow \id_{\Sp}$ be the (natural) counit of the adjunction. Then $\epsilon_n \wedge \epsilon_m : X \langle n \rangle \wedge Y\langle m \rangle \to X \wedge Y$ is natural, so 
$$
\alpha : X\langle n \rangle \wedge Y\langle m \rangle \simeq (X\langle n \rangle \wedge Y\langle m \rangle)\langle n+m \rangle \xrightarrow{\tau_{n+m}(\epsilon_n \wedge \epsilon_m)} (X \wedge Y)\langle n+m \rangle 
$$
is natural as well. Thus, we can obtain $\mu_{\RR}$ as the composite of the natural transformations $\alpha$ and $\mu_{\KU_{\RR}^e}$. 
\end{proof}

\begin{proposition}
    The Thom diagonal $\Delta: \MSpin^c \to \MSpin^c \wedge \BSpin^c$ refines to a $C_2$-equivariant map $$\Delta_{\RR}: (\MSpin^c_{\RR})^e \to (\MSpin^c_{\RR})^{e} \wedge \BSpin^c_{\RR}.$$
\end{proposition}
\begin{proof}
    Note that $\Delta$ is induced by the vector bundle maps 
$$
\begin{tikzcd}
    \ESpin^c(n) \times_{\Spin^c(n)}\RR^n \arrow[r] \arrow[d] & {(\ESpin^c(n) \times_{\Spin^c(n)}\RR^n) \times \underline{0}} \arrow[d]\\
    \BSpin^c(n) \arrow[r] & \BSpin^c(n) \times \BSpin^c(n),
\end{tikzcd}
$$
which are trivially $C_2$-equivariant maps 
$$
\begin{tikzcd}
    \ESpin^c_{\RR}(n) \times_{\Spin^c(n)}\RR^n \arrow[r] \arrow[d] & {(\ESpin^c_{\RR}(n) \times_{\Spin^c(n)}\RR^n) \times \underline{0}} \arrow[d]\\
    \BSpin^c_{\RR}(n) \arrow[r] & \BSpin^c_{\RR}(n) \times \BSpin^c_{\RR}(n),
\end{tikzcd}
$$
at the point-set level by the definition of the $C_2$-actions in \cite{HK24}. 
\end{proof}

Lastly, note that by Theorem \ref{Real.orientation}, the spin$^c$ orientation of $\KU$ refines to a $C_2$-equivariant map of ring spectra with $C_2$-action, $(\varphi^c_{\RR})^e : (\MSpin^c_{\RR})^e \to (\ku_{\RR})^e$.

\begin{definition}\label{def.equiv.fI}
    Let $I \in \mathcal{P}$. Define $f^I_{\RR} : (\MSpin^c_{\RR})^e \to \kuRangle{4|I|}^e$ to be the composite,
    $$
    f^I_{\RR} : (\MSpin^c_{\RR})^e \xrightarrow{\Delta_{\RR}} (\MSpin^c_{\RR})^e \wedge \BSpin^c_{\RR} \xrightarrow{(\varphi^c_{\RR})^e \wedge \widetilde{\pi}^I_{\RR}} \ku_{\RR}^e \wedge \ku_{\RR} \langle 4|I| \rangle^e \xrightarrow{\mu_{\RR}} \ku_{\RR} \langle 4|I| \rangle^e.
    $$
\end{definition}

\begin{proposition}
    The underlying map of $f^I_{\RR}$ of Definition \ref{def.equiv.fI} is the map 
    $f^I : \MSpin^c \to \kuangle{4|I|}$ of Definition \ref{def.fI}. 
\end{proposition}
\begin{proof}
This follows directly from the constructions above of each of the maps that compose to define $f^I_{\RR}$.
\end{proof}

\subsection{Equivariance of the $\HZmod$ components}\label{sec.mod2equiv} Next, we need to show that for each $z \in Z \subset \H^*(\MSpin^c;\Z2)$, the component, $$(f^z : \MSpin^c \to \Sigma^{|z|}\HZmod) \in \pi_{-|z|}\map_{\Sp}( \MSpin^c , \HZmod),$$ of the Anderson--Brown--Peterson map \eqref{ABP.map} refines to a $C_2$-equivariant map, $$(f^z_{\RR}:(\MSpinR)^e \to \Sigma^{|z|}\HZmod)  \in \pi_{-|z|}\map_{\Sp^{\BC}}( (\MSpin^c_{\RR})^e , \HZmod).$$ Since homotopy orbits is left adjoint to the constant diagram functor $\Sp \to \SpBC$, we can identify the mapping spectrum in $\SpBC$ as 
$$
\begin{aligned}
\map_{\Sp^{\BC}}( (\MSpin^c_{\RR})^e , \HZmod) &\simeq \map_{\Sp}((\MSpin^c_{\RR})_{hC_2},\HZmod).
\end{aligned}
$$
In other words, we need to show that each $z \in Z \subset \H^*(\MSpin^c;\Z2)$ descends to an element $z_{\RR} \in \H^*((\MSpin^c_{\RR})_{hC_2}; \Z2)$. By Proposition 2.3 of \cite{glasman15_hodgeTHH}, the equivariant mapping spaces in $\Sp^{\BC}$ can be identified as the homotopy fixed points of the mapping space of the underlying spectra under the conjugation action, 
$$
\Map_{\Sp^{\BC}}( X , Y) \simeq \Map_{\Sp}( X , Y)^{hC_2}. 
$$ 
Thus, the corresponding map of mapping spectra, 
$$
\map_{\Sp^{\BC}}( X , Y) \xrightarrow{i_{X,Y}} \map_{\Sp}( X , Y)^{hC_2},
$$ 
induces equivalences, 
$$
\Omega^{\infty - n}\map_{\Sp^{\BC}}( X , Y) \simeq \Omega^{\infty -n}\map_{\Sp}( X , Y)^{hC_2},
$$
which implies that $i_{X,Y}$ is an equivalence. Applying this to our situation, we have 
$$
\begin{aligned}
\map_{\Sp}( (\MSpin^c_{\RR})^e , \HZmod)^{hC_2} &\simeq \map_{\Sp^{\BC}}( (\MSpin^c_{\RR})^e , \HZmod) \\
&\simeq \map_{\Sp}( (\MSpin^c_{\RR})_{hC_2} , \HZmod). 
\end{aligned}
$$
Thus, the homotopy fixed point spectral sequence computing $$\pi_{*}\map_{\Sp}((\MSpinR)^e,\HZmod)^{hC_2} \cong \H^{-*}((\MSpin^c_{\RR})_{hC_2};\Z2)$$ can be written as, 
\begin{equation}\label{HFPSSmod2}
\H^*(C_2; \H^*(\MSpin^c;\Z2)) \implies \H^*((\MSpin^c_{\RR})_{hC_2};\Z2).
\end{equation}
To obtain the equivariant refinements, $f^z_{\RR}$, we will show that each $z \in Z$ survives this spectral sequence (Proposition \ref{prop.HFPSSmod2}). Lemmas \ref{Z.equiv.homotopy} and \ref{lem.trivaction.mod2} are used to give a convenient presentation of the $E_2$ page and Proposition \ref{borelku} helps determine differentials. 

\begin{lemma}\label{Z.equiv.homotopy}
    The map $f^Z = \bigvee_z f^z$ induces a $C_2$-equivariant map on homotopy groups, $f^Z_*: \pi_*\MSpin^c \to \pi_*(\bigvee_{z \in Z} \Sigma^{|z|}\HZmod)$.
\end{lemma}
\begin{proof}
    We need to show that for each $[M] \in \pi_*\MSpin^c$, we have $f^Z_*([M]) = f^Z_*(\overline{[M]})$ in $\pi_*(\bigvee_{z \in Z} \Sigma^{|z|}\HZmod)$. We will proceed using the fact that a spin$^c$ bordism class is determined by its rational characteristic numbers and its Stiefel-Whitney numbers  (\cite{Stong68}, page 337). From the exact sequences \eqref{equiv.ses} and \eqref{equiv.ses.2}, we can deduce that the impact of the $C_2$-action on the characteristic classes of $M$ are as follows: $$\begin{aligned} c_1(\overline{[M]}) &= -c_1([M]) \\ p_i(\overline{[M]}) &= p_i([M]) \\ w_i(\overline{[M]}) &= w_i([M]).\end{aligned}$$  Using the splitting \eqref{ABP.map}, we can write $[M] = [M_{\KU}] + [M_Z]$, where $f^Z(M_{\KU}) = 0$, and $M_Z$ has trivial rational characteristic numbers. 
    Thus, it is sufficient to consider the case where $[M] = [M_Z]$, so that all rational characteristic numbers of $M$ vanish. In this case, by the identities above, $\chi(\overline{[M]}) = \chi([M])$, for all Stiefel-Whitney and rational characteristic numbers $\chi$. Therefore, $\overline{[M]} = [M]$.

\end{proof}

\begin{lemma}\label{lem.trivaction.mod2}
    The $C_2$-action on $\H^*(\MSpin^c;\Z2)$ induced by the $C_2$-action on the underlying spectrum of $\MSpinR$ is trivial. 
\end{lemma}
\begin{proof}
Let $x \in \H^*(\MSpin^c;\Z2)$ be an $\mathcal{A}$-module indecomposable, where $\mathcal{A}$ is the mod 2 Steenrod algebra. By Anderson--Brown--Peterson \cite{ABPspin67}, $x$ either generates an $\mathcal{A}$ summand or an $\mathcal{A}/\!/\mathcal{E}(1)$ summand in  $\H^*(\MSpin^c;\Z2)$, where $\mathcal{E}(1)$ is the subalgebra of $\mathcal{A}$ generated by $Q_0=\Sq^1$ and $Q_1=\Sq^2\Sq^1+\Sq^1\Sq^2$. 
Suppose $x$ generates an $\mathcal{A}$ summand, then there is a 2-torsion element $w \in \pi_*(\MSpin^c)$ such that  $w \mapsto x^*$ under the Hurewicz homomorphism $\pi_*(\MSpin^c) \rightarrow \H_*(\MSpin^c;\mathbb{Z}/2)$. By Lemma \ref{Z.equiv.homotopy}, the action of $C_2$ on $w$ is trivial, and thus must also be trivial on $x^*$ and $(x^*)^* = x$. 
Now suppose $x$ generates an $\mathcal{A}/\!/\mathcal{E}(1)$ summand. Then there is an element $v \in \pi_n(\MSpin^c)$ such that $F^{c}(v)$ generates  $\pi_n(\ku\langle n\rangle)$ and $v \mapsto x^*$ under the Hurewicz homomorphism. The $C_2$-action sends $v$ to $\pm v$ (depending on the parity of $\frac{n}{4}$), hence the action on $x^*$ and $(x^*)^*=x$ is trivial, since the sign action is trivial modulo 2. 
Thus, the $C_2$-action on all $\mathcal{A}$-module indecomposables of $\H^*(\MSpin;\Z2)$ is trivial. By naturality of the Steenrod squares, it is in fact trivial on all elements of $\H^*(\MSpin^c;\Z2)$.
\end{proof}

\begin{proposition}\label{borelku}
    The mod 2 Borel cohomology of $\ku_{\RR}$ is given by $$\H^*((\ku_{\RR})_{hC_2}; \Z2) \cong \H^*(\BC;\mathbb{Z}/2) \otimes \H^*(\ku;\mathbb{Z}/2).$$
\end{proposition}
\begin{proof}
Consider the homotopy fixed point spectral sequence: 
\begin{align*}
    E_2^{s,t} = \H^{s}(C_2;\H^t(\ku;\mathbb{Z}/2) \implies \H^{s+t}((\ku_{\RR})_{hC_2}; \mathbb{Z}/2)).
\end{align*}
The $C_2$-action on $\H^*(\ku;\mathbb{Z}/2)$ is trivial (see Lemma \ref{lem.trivaction.mod2}), so the $E_2$ page is isomorphic to $\H^*(\BC;\mathbb{Z}/2) \otimes \H^*(\ku;\mathbb{Z}/2)$. To show that this spectral sequence collapses at $E_2$, we use the cofiber sequence:
\begin{align*}
    (\ku_{\RR})_{hC_2} \rightarrow \ku_{\RR}^{C_2} \rightarrow \ku_{\RR}^{gC_2}
\end{align*}
where $\ku_{\RR}^{C_2} \cong \ko$ and $\ku_{\RR}^{gC_2} \cong \bigvee_{k\geq0}\Sigma^{4k} \H\mathbb{Z}/2$ (see Corollary 3.6.2 of \cite{BrunerGreenlees10_connRealK}). The cofiber sequence induces a long exact sequence in cohomology: 
\begin{align*}
    \dots \rightarrow \H^{n-1}((\ku_{\RR})_{hC_2};\mathbb{Z}/2) \xrightarrow{f} \H^n(\bigvee_{k\geq0}\Sigma^{4k}\H\mathbb{Z}/2;\mathbb{Z}/2) \xrightarrow{g} \H^n(\ko;\mathbb{Z}/2)\rightarrow \dots 
\end{align*}
where $\H^*(\bigvee_{k\geq0}\Sigma^{4k} \H\mathbb{Z}/2;\mathbb{Z}/2) \cong \bigoplus_{k\geq0}\Sigma^{4k} \mathcal{A}$ and $\H^*(\ko;\mathbb{Z}/2) \cong \mathcal{A}/\!/\mathcal{A}(1)$. By exactness, $\dim (\H^{n-1}((\ku_{\RR})_{hC_2};\Z2)$ must be greater than or equal to $$N=\dim \H^n(\bigvee_{k\geq0}\Sigma^{4k}\H\mathbb{Z}/2;\mathbb{Z}/2) -\dim \H^n(\ko;\mathbb{Z}/2).$$ 

We will now show that $\displaystyle\sum_{s+t = n-1}\dim E_2^{s,t} = N$, which implies that there are no nonzero differentials. Let $\mathcal{G}=\bigoplus_{k\geq0} \Sigma^{4k} \mathcal{A}$. The Poincar\'e series for the corresponding cohomology rings and the $E_2$ page are given by:
\begin{align*}
    P_{\ko}(t) &= \frac{1}{(1-t^6)(1-t^4)}\prod_{k\geq3}\frac{1}{(1-t^{2^k-1})} \\
    P_{\mathcal{G}}(t) &= \frac{1}{1-t^4} \prod_{k\geq1}\frac{1}{1-t^{2^k-1}} \\
    P_{E_2}(t) &= \frac{1}{(1-t^6)(1-t^2)(1-t)}\prod_{k\geq3}\frac{1}{(1-t^{2^k-1})}
\end{align*}
Observe the following identity between $P_\mathcal{G}(t)$, $P_{\ko}(t)$, and $P_{E_2}(t)$:

\begin{align*}
    P_{\mathcal{G}}(t)-P_{\ko}(t)&=\frac{1}{1-t^4} \prod_{k\geq1}\frac{1}{1-t^{2^k-1}} - \frac{1}{(1-t^6)(1-t^4)}\prod_{k\geq3}\frac{1}{(1-t^{2^k-1})} \\
    &= \left(\frac{1}{1-t^4}\left( \frac{1}{(1-t)(1-t^3)}-\frac{1}{1-t^6}\right)\right)\prod_{k\geq3}\frac{1}{1-t^{2^k-1}} \\
    &= \left(\frac{1}{1-t^4}\left(\frac{t(1-t)(t^2+1)(t^2+t+1)}{(1-t)(1-t^3)(1-t^6)}\right)\right)\prod_{k\geq3}\frac{1}{1-t^{2^k-1}} \\
    \end{align*}
    \begin{align*}
    \phantom{P_{\mathcal{G}}(t)-P_{\ko}(t)} &= \left(\frac{t(t^2+t+1)}{(1-t^2)(1-t^3)(1-t^6)}\right)\prod_{k\geq3}\frac{1}{1-t^{2^k-1}} \\
    &=\left(\frac{t(t^2+t+1)}{(1-t^2)(1-t)(t^2+t+1)(1-t^6)}\right)\prod_{k\geq3}\frac{1}{1-t^{2^k-1}}  \\
    &= \left(\frac{t}{(1-t^2)(1-t)(1-t^6)}\right)\prod_{k\geq3}\frac{1}{1-t^{2^k-1}}= tP_{E_2}(t).
\end{align*}

 The above identity states that the coefficient of the degree $n-1$ term of $P_{E_2}(t)$ is equal to the coefficient of the degree $n$ term of $P_{\mathcal{G}}(t) -P_{\ko}(t)$. Thus, $\dim E_2^{s,t} = N$ and the spectral sequence collapses at the $E_2$ page.

\end{proof}

\begin{proposition}\label{prop.HFPSSmod2}
    Every $z \in Z \subset \H^*(\MSpin^c;\Z2)$ descends to an element $z_{\RR} \in \H^*((\MSpin^c_{\RR})_{hC_2}; \Z2)$. Thus, each component, $f^z: \MSpin^c \to \Sigma^{|z|}\HZmod$, of \eqref{ABP.map} refines to a $C_2$-equivariant map $f^z_{\RR}:(\MSpinR)^e \to \bigvee_{z \in Z} \Sigma^{|z|}\HZmod$.
\end{proposition}
\begin{proof}
    Fix an element $z \in Z$. As discussed above, we show that $z$ survives the spectral sequence \eqref{HFPSSmod2},
    $$
    \H^*(C_2; \H^*(\MSpin^c;\Z2)) \implies \H^*((\MSpin^c_{\RR})_{hC_2};\Z2).
    $$
    By Lemma \ref{lem.trivaction.mod2}, the $E_2$ page simplifies to:
    \begin{align*}
        E_{2}^{*,*} \cong \H^*(\BC;\mathbb{Z}/2) \otimes \H^*(\MSpin^c;\mathbb{Z}/2).
    \end{align*}
Since $z$ is an element of $\H^*(\MSpin^c;\Z2)$ and this is a first quadrant spectral sequence, it cannot be the target of a differential. It remains to show that $z$ does not support a differential.

\vspace{3mm}

Let $U$ be the Thom class and $w_i$ the $i$th Stiefel-Whitney class. Recall that $z$ is an $\mathcal{A}$-module indecomposable. Since $\Sq^iU=w_iU$ in $\H^*(\MSpin^c;\Z2)$, it follows that $z$ must decompose into a product of Stiefel-Whitney classes. By the Leibniz rule, differentials on $z$ are determined by the differentials on each $w_i$ factor of $z$. In order to analyze these differentials, we consider the map of spectral sequences induced by the equivariant map, $\MSpin^c_{\RR} \rightarrow \ku_{\RR}$ of Theorem \ref{Real.orientation}, 
    $$
\begin{tikzcd}
    \H^*(C_2;\H^*(\ku;\Z2)) \arrow[r, Rightarrow] \arrow[d] & \H^*((\ku_{\RR})_{hC_2};\Z2) \arrow[d] \\
    \H^*(C_2;\H^*(\MSpin^c;\mathbb{Z}/2)) \arrow[r,Rightarrow] & \H^*((\MSpin^c)_{hC_2};\mathbb{Z}/2)
\end{tikzcd}
    $$
The map of $E_2$ pages is an injection and its image includes every $w_i$ term. By Lemma ~\ref{borelku}, the top spectral sequence collapses. By naturality of the differentials, this implies the differentials on each $w_iU$ in the bottom spectral sequence are zero. Hence, $z$ survives to the $E_\infty$ page. 
\end{proof}

Thus, every component of the Anderson--Brown--Peterson map refines to a $C_2$-equivariant map in $\SpBC$, which completes the proof of Theorem \ref{equiv.ABP}.

\section{The homotopy fixed points of Real spin bordism}\label{sec.HFP}

In this section, we apply Theorem \ref{equiv.ABP} to compute the homotopy fixed points of $\MSpin^c_{\RR}$. For this, we need a few technical lemmas involving compatibility of homotopy fixed points with 2-localization and sums.

\begin{lemma}\label{2loc.hC2.commute}
Let $X$ be a spectrum with $C_2$-action whose homotopy groups are finitely generated in each degree, and let $(\:\:)_{(2)} : \Sp \to \Sp_{(2)}$ denote $2$-localization. Then $$(X^{hC_2})_{(2)} \simeq (X_{(2)})^{hC_2}.$$
\end{lemma}

\begin{proof}
    Taking homotopy fixed points of the map $X \to X_{(2)}$ yields a map $X^{hC_2} \to (X_{(2)})^{hC_2}$. Since $2$-localization is a left Bousfield localization, $(X_{(2)})^{hC_2}$ is $2$-local, so we get an induced map $(X^{hC_2})_{(2)} \to (X_{(2)})^{hC_2}$. This induces a map of homotopy fixed point spectral sequences,
    $$
\begin{tikzcd}
    \H^*(C_2;\pi_*X)_{(2)} \arrow[r, Rightarrow] \arrow[d] & \pi_*(X^{hC_2})_{(2)} \arrow[d] \\
    \H^*(C_2;\pi_*X_{(2)}) \arrow[r,Rightarrow] & \pi_*((X_{(2)})^{hC_2}),
\end{tikzcd}
    $$
    Since the filtrations are bounded below, both spectral sequences converge. The map induces an isomorphism on the $E_2$-page, since $2$-localization commutes with both $\pi_*$ and $\H^*(C_2;-)$.  Thus, we get an isomorphism of $E_{\infty}$-pages, and by Theorem 8.2 of Boardman \cite{Boardman1999}, an isomorphism on the abutment. So, the map $(X^{hC_2})_{(2)} \to (X_{(2)})^{hC_2}$ is an  equivalence.
\end{proof}

\begin{lemma}\label{lem.equiv.2loc}
    Let $X,Y$ be spectra with $C_2$-action whose homotopy groups are finitely generated in each degree. If $f:X \to Y$ is a map in $\SpBC$ which is a 2-local equivalence on underlying spectra, then $f$ induces a 2-local equivalence on homotopy fixed points.
\end{lemma}
\begin{proof}
First, note that postcomposing $X,Y:\BC \to \Sp$ with the 2-localization functor, $(\:\:)_{(2)} : \Sp \to \Sp$, yields  $X_{(2)}, Y_{(2)} \in \SpBC$.  Similarly, applying $(\:\:)_{(2)}$ to $f$ yields a $C_2$-equivariant map  $f_{(2)}:X_{(2)} \to Y_{(2)}$ in $\SpBC$. Since by assumption, $f_{(2)}$ is an equivalence on underlying spectra, it induces an equivalence $(f_{(2)})^{hC_2}:(X_{(2)})^{hC_2} \xrightarrow{\sim} (Y_{(2)})^{hC_2}$. By Lemma \ref{2loc.hC2.commute}, this gives the desired equivalence, $f^{hC_2}_{(2)}:(X^{hC_2})_{(2)} \xrightarrow{\sim} (Y^{hC_2})_{(2)}$.
\end{proof}

\begin{lemma}\label{lem.sum.prod}
    Let $E_k$ be a connective spectrum for $k \in \ZZ_{\geq 0}$, and let $\{n_k\}$ be a sequence of integers with $\lim_{k \to \infty} n_k = \infty$. Then 
    $$ 
    \bigvee_{k \in \ZZ_{\geq 0}} \Sigma^{n_k} E_k \simeq \prod_{k \in \ZZ_{\geq 0} } \Sigma^{n_k} E_k.
    $$
\end{lemma}
\begin{proof}
    The canonical map $f: \bigvee_{k \in \ZZ_{\geq 0}} \Sigma^{n_k} E_k \rightarrow \prod_{k \in \ZZ_{\geq 0} } \Sigma^{n_k} E_k$ induces a map of homotopy groups
    \begin{align*}
        \pi_n( \bigvee_{k \in \ZZ_{\geq 0}} \Sigma^{n_k} E_k) &\cong \bigoplus_{k\in \ZZ_{\geq 0}} \pi_n \Sigma^{n_k} E_k \\
         &\cong \bigoplus_{k\in \ZZ_{\geq 0}} \pi_{n-n_k} E_k
         \longrightarrow \prod_{k \in \ZZ_{\geq 0} } \pi_n( \Sigma^{n_k} E_k ) \cong \prod_{k \in \ZZ_{\geq 0} } \pi_{n-n_k} E_k 
    \end{align*}
Since $E_k$ is connective, $\pi_{n-n_k} E_k$ is nonzero for only finitely many $n_k$ for each fixed $n$. Thus, $f$ induces an isomorphism on all homotopy groups, giving the desired result. 
\end{proof}

\begin{corollary}[Theorem \ref{homotopy.fixedpoints}]\label{cor.homotopy.fixedpoints}
The Anderson--Brown--Peterson map \eqref{Real.ABP.map} induces a 2-local equivalence, 
$$
(\MSpinR)^{hC_2} \to \bigvee_{I \in \mathcal{P}}\kuRangle{4|I|}^{hC_2} \vee \bigvee_{z \in Z}\Sigma^{|z|}\HZmod^{hC_2}.
$$
\end{corollary}
\begin{proof}
    By Lemma \ref{lem.equiv.2loc}, the map in \eqref{Real.ABP.map} induces a 2-local equivalence on homotopy fixed points. Since taking homotopy fixed points commutes with taking products, Lemma \ref{lem.sum.prod} allows us to express the homotopy fixed points of the right hand side in terms of the homotopy fixed points of each of the summands. 
\end{proof}

Next, we apply Corollary \ref{cor.homotopy.fixedpoints} to identify the homotopy groups of $(\MSpinR)^{hC_2}$. First, we compute the homotopy groups of each of the summands.

\begin{proposition}\label{HFP.ku.covers}
   If $2n = 8k+r$, for $r = 0,2, \text{ or } 4$, then,
   $$ \pi_*\ku_{\RR}\langle2n\rangle^{hC_2} \cong \pi_*\ko\langle2n\rangle \bigoplus_{m\geq1} \mathbb{Z}/2\{\delta^m\}
   $$
   where $|\delta^m| = 8k+\frac{r}{2}-4m$. 
   
\end{proposition}
\begin{proof}
    Following Example 3.2.2 in \cite{BrunerGreenlees10_connRealK}, consider the homotopy fixed point spectral sequence
    \begin{align*}
        E^{s,t}_{2} = \H^{-s}(C_2,\pi_{t}\ku) \Rightarrow \pi_{s+t}\ku_{\RR}^{hC_2}
    \end{align*}
Denote the generator of $C_2$ as $\tau$, then the action of $\tau$ on $\pi_*\ku \cong \mathbb{Z}[\nu]$, where $|\nu|=2$, is $\tau(\nu)= -\nu$. It follows that the $E_2$ page has a presentation 
\begin{align*}
    E_{2}^{s,t} = \mathbb{Z}[y,z]/(2y,2z) \otimes \mathbb{Z}[\nu^2]
\end{align*}
where $y\in E_2^{-2,0}$, $\nu^2\in E_{2}^{0,4}$, and $z\in E_{2}^{-1,2}$. It is shown in \cite{BrunerGreenlees10_connRealK} that the spectral sequence collapses at $E_4$ and $E_{\infty}^{0,*}=\mathbb{Z}[2\nu^2,\nu^4]$, $E_{\infty}^{-1,*}=z \mathbb{Z}/2[\nu^4]$, $E_{\infty}^{-2,*}=y\nu^2\mathbb{Z}/2[\nu^4]$, and $y^{2i}\in E_{\infty}^{4i,0}=$ for $i > 0$. Denote $\delta$ as the element in homotopy detected by $y^{2}$, then this proves the result for $2n=0$. 

\vspace{2.5mm}

The $E_2$ page for higher connective covers is obtained from the $E_2$ page above by setting all entries below the $t=2n$ line equal to zero. The result follows from keeping track of the bidegree of generators on the $t=2n$ line that are no longer the target of a differential and therefore survive to $E_\infty$.
\end{proof}


\begin{proposition}
    \begin{align*}
        \pi_*((\H\mathbb{Z}/2)^{hC_2}) &\cong \mathbb{Z}/2[w] \\
        &\cong \H^{-*}(\BC;\Z2),
    \end{align*}
    where $|w|=-1$.
\end{proposition}
\begin{proof}
    Consider the homotopy fixed point spectral sequence,
 \begin{align*}
     E_2^{s,t} = \H^s(C_2;\pi_t\H\mathbb{Z}/2) \implies \pi_{t-s}(\H\mathbb{Z}/2)^{hC_2}.
 \end{align*}
 Since $\pi_t\H\mathbb{Z}/2 \cong \mathbb{Z}/2$ for $t=0$ and is trivial otherwise, the $E_2$-page is isomorphic to $\mathbb{Z}/2[w]$ where $w$ is in bidegree $(1,0)$. There is no room for differentials and the spectral sequence collapses. Alternatively, since the $C_2$-action on $\HZmod$ is trivial, 
 $$
 \begin{aligned}
     \pi_*((\H\mathbb{Z}/2)^{hC_2}) &\cong \pi_*(\map_{\Sp^{\BC}}(\E\!C_2, \HZmod)) \\
     &\cong \pi_*(\map_{\Sp}(\BC,\HZmod)) \\
     &\cong \H^{-*}(\BC;\Z2).
 \end{aligned}
 $$
\end{proof}

Let $\Z2\{a\}$ denote a copy of $\Z2$ generated by an element $a$ .

\begin{corollary}[Theorem \ref{thm.HFP.groups}]\label{cor.homotopy.groups}
     There exists an isomorphism of abelian groups,
     $$
     \begin{aligned}
     \pi_*((\MSpinR)^{hC_2}) &\cong \pi_*\big(\bigvee_{I}\kuRangle{4|I|}^{hC_2} \vee \bigvee_{z \in Z} \Sigma^{|z|}\HZmod^{hC_2}\big) \\
     &\cong \bigoplus_{I \in \mathcal{P}} (\pi_*\ko\langle 4 |I| \rangle \oplus \bigoplus_{m\geq1} \mathbb{Z}/2\{\delta^m_I\}) \oplus \bigoplus_{z \in Z, n \geq 1} \Z2\{w^n_z\}, 
     \end{aligned}
     $$
    where  $|\delta^m_I| = 4|I| - 4m$ when $|I|$ is even,  $|\delta^m_I| = 4|I| - 2 - 4m$ when $|I|$ is odd, and $|w^n_z| = |z| - n$.
\end{corollary}
\begin{proof}
    First note that since $\pi_*(\MSpin^c)$ has no odd torsion, the homotopy fixed point spectral sequence implies that $\pi_*((\MSpinR)^{hC_2})$ also has no odd torsion. Then since the homotopy groups of both sides are finitely generated in each degree, the existence of the 2-local equivalence in Corollary \ref{cor.homotopy.fixedpoints} implies the existence of the desired isomorphism. 
\end{proof}


\section{Towards a genuine splitting of Real spin bordism}\label{sec.genuine}

The observation in Remark \ref{rem.noliftko} suggests that the Anderson--Brown--Peterson splitting does not refine to a genuine splitting in the naive expected way. The following proposition makes this more concrete. 

\begin{proposition}\label{prop.nonexist.genuine}
    When $|I|$ is odd, the $C_2$-equivariant map $\widetilde{\pi}^I_{\RR}: \BSpin^c_{\RR} \to \kuRangle{4|I|}^e$ does not refine to a genuine $C_2$-map $\B_{\RR} \to \kuRangle{4|I|}$, where $\B_{\RR}$ is the base space of any Real $\Spin^c$-bundle whose underlying $\Spin^c$-bundle is universal. 
\end{proposition}
\begin{proof}
    Since $\Spin(n) \to \Spin^c(n)^{C_2}$, the $C_2$-fixed points of any such $C_2$-space $\B_{\RR}$ receives a natural map from $\BSpin$ factoring the usual map $\BSpin \to \BSpin^c$. If such a map $\B_{\RR} \to \kuRangle{4|I|}$ existed, taking $C_2$-fixed points would yield a commutative diagram,
    $$
    \begin{tikzcd}[row sep = 7mm]
        \BSpin \arrow[r] \arrow[d] \arrow[rr,bend left = 22, "\pi^I_r"]& \ko\langle 4|I| \rangle \arrow[d, "\simeq"] \arrow[r] &\KO \arrow[dd] \\
        \B_{\RR}^{C_2}  \arrow[d] \arrow[r] & \kuRangle{4|I|}^{C_2} \arrow[d] \\
        \BSpin^c \arrow[r,"\widetilde{\pi}^I"] \arrow[rr, bend right = 22, "\pi^I"'] & \ku\langle 4|I| \rangle \arrow[r] & \KU.
    \end{tikzcd}
    $$
    But there does not exist a map $\BSpin \to \ko\langle 4|I| \rangle$ lifting $\pi^I_r$ (see \cite{Stong68}).
\end{proof}

By Proposition \ref{rem.noliftko}, the particular construction by Anderson--Brown--Peterson of the map $f^I: \MSpin^c \to \kuangle{4|I|}$ does not refine to a map $\MSpin^c_{\RR} \to \kuRangle{4|I|}$ of genuine $C_2$-spectra, which leads us to believe that the corresponding naive guess for a genuine splitting of Real spin bordism does not hold. 

\begin{conjecture}\label{conj.nongenuine.split}
    There does not exist a map of genuine $C_2$-spectra,
    $$
    \MSpin^c_{\RR} \to \big ( \bigvee_{I \in \mathcal{P}} \kuRangle{4|I|} \big ) \vee Z,
    $$
    whose induced map on underlying spectra is the Anderson--Brown--Peterson map \eqref{ABP.map}.
\end{conjecture}

Instead, we propose a different candidate for a genuine refinement of the equivariant splitting of Section \ref{sec.equiv.ABP}.

\begin{proposition}\label{prop.new.genuine.kr}
    For odd $n$, there exists a genuine $C_2$-spectrum, $\kuRangle{4n, 2}$, such that $$\kuRangle{4n, 2}^e \simeq \kuRangle{4n}^e \;\;\;\text{and}\;\;\; \kuRangle{4n, 2}^{C_2} \simeq \ko\langle 4n-2 \rangle,$$ whose restriction map $\mathrm{res}: \ko \langle 4n-2 \rangle \to \kuangle{4n}$ is the lift $c$ in \eqref{lift.ko42}.
\end{proposition}
\begin{proof}
    First, recall the Tate diagram for $\kuRangle{4n}$, 
    $$
    \begin{tikzcd}[column sep = 4mm, row sep = 9mm]
        \kuRangle{4n}_{hC_2} \arrow[rrr] \arrow[d,"=\:"'] &&& \kuRangle{4n}^{C_2} \simeq \ko\langle 4n \rangle \arrow[rrr] \arrow[d]  \arrow[drrr, phantom, "\ulcorner", near end] &&& \kuRangle{4n}^{gC_2} \arrow[d] \\
        \kuRangle{4n}_{hC_2} \arrow[rrr] &&& \kuRangle{4n}^{hC_2}  \arrow[rrr] &&& \kuRangle{4n}^{tC_2},
    \end{tikzcd}
    $$
    where the top and bottom rows are cofiber sequences. Motivated by this, define a geometric fixed point spectrum for $\kuRangle{4n,2}$ by $$\kuRangle{4n,2}^{gC_2} : = \text{cofib}(\kuRangle{4n}_{hC_2} \to  \ko\langle 4n \rangle \to  \ko\langle 4n -2 \rangle ).$$ Then notice that by \eqref{lift.ko42}, we have a commutative diagram,
    $$
    \begin{tikzcd}[column sep = -5mm, row sep = 9mm]
        \kuRangle{4n}_{hC_2} \arrow[rr] \arrow[d,"=\:"'] &\phantom{----}& \ko\langle 4n \rangle \arrow[rr] \arrow[dr] && \ko\langle 4n -2 \rangle \arrow[rr] \arrow[dl,"c"] \arrow[drr, phantom, "\ulcorner", near end] &\phantom{----}& \kuRangle{4n, 2}^{gC_2} \arrow[d] \\
        \kuRangle{4n}_{hC_2} \arrow[rrr] &&& \kuRangle{4n}^{hC_2}  \arrow[rrr] &&& P \simeq \kuRangle{4n}^{tC_2},
    \end{tikzcd}
    $$
    where $P := \kuRangle{4n}^{hC_2} \coprod_{\ko\langle 4n -2 \rangle} \kuRangle{4n, 2}^{gC_2}$. By Theorem 3.21 (Example 3.29) of \cite{Glasman17}, the triple, $$(\kuRangle{4n}^e, \kuRangle{4n, 2}^{gC_2}, \kuRangle{4n, 2}^{gC_2} \to \kuRangle{4n}^{tC_2}),$$ determines a genuine $C_2$-spectrum, $\kuRangle{4n, 2}$, which has the desired property by construction. 
\end{proof}

\begin{proposition}\label{prop.genuine.lift}
    The map $\widetilde{\pi}^I_{\RR} : \BSO \to \kuRangle{4|I|}^e$ in (the proof of) Proposition \ref{prop.equiv.lift.cover} refines to a map $ \BSO \to \kuRangle{4|I|,2}$ of genuine $C_2$-spectra.
\end{proposition}
\begin{proof}
    This follows directly from the existence of the classical lift $$\widetilde{\pi}^I : \BSO \to \ko\langle 4|I| -2 \rangle$$ of Anderson--Brown--Peterson\cite{ABPspin67} and Proposition \ref{prop.new.genuine.kr}.
\end{proof}

Pulling back the genuine equivariant lift of Proposition \ref{prop.genuine.lift} to a genuine refinement of $\BSpin^c_{\RR}$ would involve a more careful analysis of the Real $\Spin^c(n)$-bundles of Section \ref{sec.prelim}. In particular, the equivariant map $\BSpin^c_{\RR} \to \BSO$ does refine to a genuine equivariant map $|\mathcal{B}\Spin^c| \to \BSO$, and hence gives an equivariant map $|\mathcal{B}\Spin^c| \to \kuRangle{4|I|,2}$ in $\Sp^{C_2}$. However, it is not immediately clear if the map $\B_{\mathrm{J}} \to \BSO$ is an equivariant map of genuine $C_2$-spaces, which would be necessary in order to refine the rest of the construction in Section \ref{sec.equiv.ABP} to the genuine setting. 

\begin{question}
Does there exist a genuine splitting of $\MSpin^c_{\RR}$ whose summands consist of $\kuRangle{4n}$, $\kuRangle{4n,2}$, and suspensions of mod 2 Eilenberg--Mac Lane spectra?
\end{question}

\bibliographystyle{plain} 
\bibliography{main}

\end{document}